%% file: main.tex
\documentclass{article}
\input{my_preamble}
\input{my_macros}

\title{On the cyclotomic Iwasawa invariants of Elliptic Curves of rank one}
\date{}
\author{Foivos Chnaras}

\begin{document}

\maketitle 

\begin{abstract}
For elliptic curves of rank 1 over $\Q$, we provide a numerical criterion that determines whether the Iwasawa invariants at a prime $p$ of good reduction attain their minimum possible value. 
\end{abstract}

\section{Introduction}
The Iwasawa theory of elliptic curves provides deep insights on their arithmetic. For an elliptic curve $E$ over $\Q$, the study of Iwasawa theory for primes of good ordinary reduction was initiated by Mazur \cite{mazur1972rational}) and Greenberg \cite{greenberg1999iwasawa}. The supersingular case was first studied by Perrin Riou \cite{perrin1984arithmetique}. Pollack introduced the signed p-adic L-functions on the analytic side (\cite{pollack2001p}) and Kobayashi introduced the signed Selmer groups on the algebraic side (\cite{kobayashi2003iwasawa}). 
The goal of our study is to explore properties of the Iwasawa invariants in both the orginary and supersingular case. 

In particular, we develop an explicit numerical criterion that determines whether for an elliptic curve $E$ over $\Q$ of rank 1 and a prime $p$ of good reduction, the Iwasawa invariants attain their minimum possible value.

\subsection{The good ordinary case}

We start with an elliptic curve $E$ over $\Q$. Assume that $E$ has good, ordinary reduction at an odd prime $p$. Let $\lambda_p(E), \mu_p(E)$ be the Iwasawa invariants associated with the Pontryagin dual of $\Sel_{p^\infty}(E/\Q_\infty)$, the $p^\infty$-Selmer group of $E$ over the cyclotomic $\Z_p$- extension $\Q_\infty/\Q$. 

It is expected for density one of the good ordinary primes that $\mu_p(E)=0$ and $\lambda_p(E)=\rank E(\Q)$. When $E$ has rank $0$ over $\Q$, then this is known by a theorem of Greenberg ({\cite[Theorem~4.1]{greenberg1999iwasawa}}):

\begin{theorem*}[Greenberg] 
Assume $E$ is an elliptic curve over a number field $K$ with good ordinary reduction at all the primes above $p$. Assume also that $\Selp(E/K)$ is finite. Then 
$$f_E(0) \sim  \frac{ \prod\limits_{v \text{ bad}}c_v \cdot \prod\limits_{v|p}N_v^2 \cdot \# \Sel(E/K)}{E(K)_{tors}^2}$$
where $N_v:=\# \tilde E(\mathbb F_v)$ the number of points of the elliptic curve $\mod v$ and where $c_v$ are the Tamagawa numbers. Finally, $f_E(T)$ is the characteristic polynomial of the Pontryagin dual of $\Selp (E/K_\infty)$.
\end{theorem*}

The number of primes that divide the Tamagawa numbers, the Selmer group and $E(K)_{tors}$ is finite and a Cebotarev type argument shows that the number of anomalous primes, i.e. the primes that divide $N_p$, has density 0. For more details, see \cite{greenberg1999iwasawa}

For elliptic curves of rank $1$, works of Perrin-Riou and Schneider (see \cite{perrin1984arithmetique}, \cite{schneider1983iwasawa}) produce the generalization of the above theorem: 

\begin{theorem}\label{Theorem 2.4}
Let $E/\Q$ be an elliptic curve with good ordinary reduction at $p>2$. Then $\ord_{T=0}f_E(T) \geq \rank E(\Q)$.Furthermore, if:
\begin{enumerate}
    \item $\Reg_p(E/\Q)\neq 0$
    \item $\Sh(E/\Q)[p^\infty]$ is finite
    \item The kernel and cokernel of the map $\Phi_M:M^\Gamma\to M_\Gamma$ are finite, where $M=\Selp(E/\Q_\infty)$.
    \end{enumerate}
    then: 
    \begin{enumerate}
        \item $\ord_{T=0}f_E(T)=\rank E(\Q)$
        \item Write $f_E(T)=T^rg(T)$, where $r:=\rank E(\Q)$. Then  $$g_E(0) \sim   \frac{\text{Reg}_p(E/\Q)}{p^r}\cdot \frac{ \prod\limits_{\ell \text{ bad}}c_\ell \cdot N_p^2 \cdot \# \Sh(E/\Q)}{ (\#E(\Q)_{tors})^2} $$ up to a $p-$adic unit.
    \end{enumerate}

\end{theorem}
A few remarks are in order: 
\begin{itemize}
    \item The assumptions (1)-(2) are conjecturally always true. The first one in particular is proven for CM elliptic curves of rank 1, by the work of Bertrand (see \cite{bertrand1980valeurs}).
    \item The third hypothesis is true if $T^2 \nmid f_E(T)$, i.e. if the rank is less or equal to 1 (see \cite[Lemma~2.11]{zerbes2009generalised}. In that case, the quotient $\chi_T(\Gamma,M):=\frac{\# \ker \Phi_M}{\# \Coker \Phi_M}$ is called the \textit{truncated Euler characteristic}. When defined, it is always an integer.
    \item In general, $\lambda_p(E)\geq \rank E(K)$ by an application of Mazur's Control Theorem. $\mu_p(E)$ is expected to be $0$ any time $E[p]$ is irreducible by Greenberg's Conjecture. It is known that $\mu_p=0$ for almost all primes (see \cite{gajek2024bound} and \cite{chakravarthy2024iwasawa} )
    
    \item When $E/\Q$ has rank $1$, we have $\mu_p(E)=0$ and $\lambda_p(E)=1 \iff g_E(0) \in \Z_p^\times.$ 
\end{itemize}

Therefore, in the rank $1$ case, the only addition is the appearance of the $p-$adic regulator.  
Our aim is to simply translate the property of $g_E(0)$ being a unit to a numerical criterion. Before we state our main results, let us first establish some notations: 

For an elliptic curve $E/\Q$ with a fixed Weierstrass model, we define:
\begin{definition}
\begin{enumerate}[]
    \item $E^0(\Q_p)=\{P \in E(\Q_p): P\text{ has non-singular reduction at } p\}$
    \item $E^0(\Q)=\{P\in E(\Q): P \text{ has non-singular reduction at every prime } p \}$
    \item $ E^k(\Q_p)=$ the $k$-th filtration on the formal group
    \item $N_p(E)=\# \tilde E_{ns}(\F_p)$ 
    \item $C_E(P)=\min\{n: nP \in E^0(\Q)\}$
    \item $C_E=\min\{n: nP\in E^0(\Q) \text{ for all } P\in E(\Q)\}$
    \item $c_p$= the $p$-th Tamagawa number
    \item $\mathrm{Tam}(E)=$ the least common multiple of the Tamagawa numbers
    \item If $P=(\frac{a}{d^2},\frac{b}{d^3}) \in E(\Q)$, then denote $a(P)=a, \quad b(P)=b, \quad d(P)=d$ and  $$nP=\left (\frac{a_n}{d_n^2},\frac{b_n}{d_n^3} \right )$$
    
\end{enumerate}
\end{definition}Our first main result is as follows.

\begin{theorem}\label{ordinary theorem}
Let $E$ be an elliptic curve of rank 1 over $\Q$, with good ordinary reduction at $p\geq 5$. \\
Let $P\in E(\Q)$ be a generator of the free part of the Mordell Weil group. \\
Let $$Q=N_p(E)\cdot C_E \cdot P \in E^1(\Q_p)\cap E^0(\Q)$$ where we can replace $C_E$ with $Tam(E)$ if the curve is given in the global minimal model.\\
Assume $p$ does not divide $C_E, \# E(\Q)_{tors}$ and $\#\Sh(E/\Q)$. \\Then  
$$\left [\mu_p(E)=0 \text{ and } \lambda_p(E)=1 \right ]\iff  a(Q)^{p-1}\not\equiv 1 \mod p^2 $$

\end{theorem}

This is comparable to Gold's criterion for the cyclotomic $\lambda$ invariant of imaginary quadratic fields (\cite{gold1974nontriviality}):

\begin{theorem*}[Gold's Criterion]
    Let $K$ be an imaginary quadratic field, and $p>3$ be a prime that splits in $K$ as $p\ot_K=\mathfrak p \bar{ \mathfrak p}$.\\
    Assume $p$ does not divide $h_K$ (the class number of $K$). Write $\mathfrak p^{h_K}=(a)$. Then $$\mu_p(K)+\lambda_p(K)=1 \iff a^{p-1}\not\equiv 1 \mod \bar{\mathfrak p}^2$$

\end{theorem*}
Note that for imaginary quadratic fields we already know $\mu_p(K)=0$ by the Ferrero-Washington theorem, although that fact is not needed for the proof of the theorem. Both ours and Gold's criteria can be made more explicit, relaxing the assumption about the non-divisibility by $p$. This will be made clear in the proof of the theorem, but for simplicity, we skip that formulation here.

\subsection{The supersingular case}
For supersingular primes, the situation is more involved.  Perrin-Riou constructed a p-adic $L$-series $L_p(E,T)$ with coefficients in $D_pE\otimes Q_p[[T]]$, where $D_pE$ is the Dieudonne model, which is a $\Q_p$-vector space of dimension 2.

There's also a construction of a $p-$adic regulator of the elliptic curve with values again in $D_pE$. Then the $p-$adic Birch and Swinnerton Dyer conjecture for the supersingular case, as posed by Perrin-Riou and Bernardi in \cite{bernardiperrinriou1993variante}, is as follows:

\begin{conjecture}[Bernardi--Perrin-Riou]\label{conjecture bsd ss}
    Suppose E is an elliptic curve over $\Q$ with good supersingular reduction at $p$. 
    \begin{enumerate}
        \item The order of vanishing of $L_p(E,T)$ at $T=0$ is equal to the rank $r$ of the elliptic curve over $\Q$.
        \item The leading term $L_p^*(E,0)$ satisfies 
        \begin{equation}
            (1-\phi)^{-2}L_p^*(E,0)=\frac{\prod_v c_v \cdot \Sh (E/Q)}{|E(\Q)_{tors}|^2} \cdot \frac{\Reg_p(E/\Q)}{p^r}
        \end{equation}
        where $\phi$ is the Frobenius as explained in Section \ref{Section Dieudonne}.
    \end{enumerate}
\end{conjecture}

The signed $p-$adic $L-$functions constructed by Pollack are related to the one constructed by Perrin Riou. That relationship leads us to our second main theorem:

\begin{theorem}\label{supersingular theorem}
 Let E be an elliptic curve over $\Q$ of rank 1 and $p$ be a prime of supersingular reduction. Let $P$ be a generator of the free Mordell Weil group and $Q=C_E \cdot N_p \cdot P$. Write $Q=(\frac{a}{d^2}, \frac{b}{d^3})$. Assume $p$ does not divide $C_E$ and $\# \Sh(E/\Q)$ and assume conjecture \ref{conjecture bsd ss} . Then: 
 
 \begin{enumerate}
    \item \begin{equation}
        \lambda_p^+(E)=1 \iff a(Q)^{p-1} \not\equiv 1 \mod p^2 
    \end{equation}
    \item If $\lambda_p^+(E)>1$ then 
    \begin{equation}\label{minus equivalence}
        \lambda^-_p(E)=1 \iff d(Q) \not\equiv 0\mod p^2
    \end{equation}
    
 \end{enumerate}

\end{theorem}

\subsection{Non-minimal models}

Finally, as it seems to have been overlooked in the literature so far, we make the following remark:

By definition, if the elliptic curve is $p-$\textbf{minimal}, then $$[E(\Q_p):E^0(\Q_p)]=c_p.$$ However, if the curve is not minimal, the value of the index will change:
\begin{theorem} \label{tamagawa theorem}
 If $E/\Q$ is not $p-$ minimal, let $E_{\min}$ be a $p$-minimal model and $\Delta, \Delta_{\min}$ the discriminants of $E$ and $E_{\min}$ respectively. Let $k=\frac{v_p(\Delta)-v_p(\Delta_{\min})}{12}$, where $v_p$ is the $p-$adic valuation. Then,
  $$E^n(\Q_p) \simeq E_{\min}^{n+k}(\Q_p)$$
In particular if $k\geq 1$: $$[E(\Q_p):E^0(\Q_p)]=[E_{\min}(\Q_p):E^k_{\min}(\Q_p)]=c_pN_pp^{k-1}$$
\end{theorem}

In practice, this says that in general, the value of $C_E$ will not be equal to $\mathrm{Tam}(E)$ and the above theorem provides an algorithm to explicitly compute it. On the algorithm used for the computation of the p-adic height in \cite{mazur2006computation} (page 14, Algorithm 3.4), this remark seems to be overlooked. This had led to a wrong implementation of the algorithm in the software SAGE, which has recently been corrected.

If one specializes this theorem to the case of the Short Weierstrass model, one gets the following:

\begin{corollary}\label{cor short weier}
    Let $E: y^2=x^3+Ax+B$ be an elliptic curve over $\Q$ given in a \textbf{short} Weierstrass model.
    Let $$m=\rm{lcm}(N_2(E),N_3(E))$$
    Then  $[E(\Q):E^0(\Q)]$ divides $m\cdot Tam(E)$.
\end{corollary}

We remark that computationally for a fixed point $P$, one can compute $C_E(P)$ instead of $C_E$.

\begin{proposition}\label{prop u_2(P) is unit}
Fix a prime $p$. Let $E/\Q$ be an elliptic curve with fixed Weierstrass model and $P\in E(\Q)$. Let $P=(\frac{a}{d^2},\frac{b}{d^3})$ and write $u_n=\frac{d^{n^2}f_n}{d_n}$, where $f_n$ is the $n-$th division polynomial. Then $$P\in E^0(\Q_p) \iff u_2\in \Z_p^\times $$
\\ In particular $$C_E(P)=\min\{n: |u_2(nP)|=1\}$$
\end{proposition}

\subsection{Overview}
In Section 2, we provide some examples to illustrate the main theorems. In Sections 3, 4 and 5 we provide the necessary background regarding the $p-$adic regulator, the $p-$adic $L$-series and the main conjecture respectively. In Section 6, we prove Theorems \ref{ordinary theorem} and \ref{supersingular theorem} and we provide some numerical data on Section 7. Finally, Section 8 is devoted to the proof of Theorem \ref{tamagawa theorem}.

\section{Examples}

We first provide a couple of examples to show the above theorems in practice. We used SAGE (\cite{SageMath}) for the computations.

\begin{example}
\end{example}
    Let $E:y^2+y=x^3-x$ which is rank 1 over $\Q$, given in the global minimal model and which has label $"37a1"$ in the Cremona table. Then using Sage, we find that $$Tam(E)=1$$ and $P=(0,-1)$ is a generator of the free Mordell-Weil group. Since\\  $\mathrm{Tam}(E)=1$, we have $$E(\Q)=E^0(\Q).$$

    The corresponding short Weierstrass form of $E$ is $E_1:y^2=x^3-16x+16$ with $P_1=(0,-4) \in E_1$ the generator. Reducing at the prime $2$ one gets $\tilde E_1:y^2=x^3$ which is clearly not smooth and $\tilde P_1=(0,0)$, which is the singular point. Therefore $$P_1\not \in E_1^0(\Q)$$

    This agrees with Proposition \ref{prop u_2(P) is unit}, as $u_2(P_1)=-8$ which is divisible by $2$. In fact, this suggests that $2$ is the only prime in which $P_1$ has singular reduction. Further, we find that $$u_2(5P_1)=1$$ which shows that $5P_1$ has non-singular reduction everywhere. In fact, one can compute that $$[E(\Q):E^0(\Q)]=N_2(E)=5$$
    which verifies Corollary \ref{cor short weier}

    For the same elliptic curve on its minimal model, if we consider the prime $p=13$, we find on the LMFDB table that $\lambda_{13}(E)=3$. On the other hand, we compute $N_{13}=16$ and let $$Q=16P=\left (\frac{a}{d^2}, \frac{b}{d^3} \right )$$ 
    We compute that $a(Q)\equiv -23 \mod p^2$ and it indeed turns out that $$23^{12}\equiv 1 \mod 13^2$$ which verifies that $\mu_{13}(E)+\lambda_{13}(E)>1$ by Theorem \ref{ordinary theorem}.

    \begin{example}
    \end{example}
    In light of the formula in Theorem \ref{Theorem 2.4}, one might expect that if $p$ divides $N_p$, then $\mu_p(E)+\lambda_p(E)>1$. The following example shows that this is not always the case. This happens because the term $N_p^2$ in the formula gets cancelled with $\Reg_p(E/\Q)$, as can be seen in the proof of Theorem \ref{ordinary theorem}.
\\
    We take the elliptic curve $E:y^2 + y = x^3 + x^2$. Then $p=5$ is a good ordinary prime which does not divide $\Tam(E), |E(\Q)_{tors}|$ or $|\Sh(E/\Q)|$. We find that $N_5=10$ but $\mu_5(E)=0$ and $\lambda_5(E)=1$ .

\section{p-adic Regulator}
To describe the construction of the $p$-adic regulator, we first need to start with the $p$-adic sigma function. There are many ways to define it and we will only present a couple of them. We follow \cite{mazurtate1991}:

\subsection{p-adic sigma function}

\begin{theorem}[Mazur-Tate]\label{Theorem 4.3}
Let $E/\Q$ be an elliptic curve with good ordinary reduction at a prime $p$ and Weierstrass equation $E:y^2+a_1xy+a_3y=x^3+a_2x^2+a_4x+a_6$. The $p-$adic sigma function is the unique function $\sigma_p:E^1(\Q_p)\cong \hat E(\mathfrak p)\to \Q_p$ satisfying any one of the following properties:
\begin{enumerate}

    \item If $m\in \Z$ and $Q \in \hat E(\bar \Z_p)$, then 
    $$\sigma_p(mQ)=\sigma_p(Q)^{m^2}f_m(Q)$$ where $f_m$ is the $m-$th division polynomial relative to the invariant differential $\omega$.
    
    \item $\sigma_p$ is odd function with $\sigma_p(t)\in t\Z_p[[t]]$ and there is a constant $c\in \Z_p$ such that the following differential equation is satisfied: $$x(t)+c=-\frac{d}{\omega}\left(\frac{1}{\sigma_p}\frac{d\sigma_p}{\omega}\right)$$
    In fact the constant $c$ can be explicitly computed and is equal to $$c=\frac{a_1^2+4a_2}{12}-\frac{1}{12}\mathbb E_2(E,\omega)$$ where $\mathbb E_2(E,\omega)$ is the value of the \textit{Katz $p$-adic weight 2 Eisenstein series at} $(E,\omega)$
\end{enumerate}
\end{theorem}
We shall see later that property (1) is also shared by the denominator of the point $P$ and follows from Proposition \ref{prop u_2(P) is unit}, which can be rephrased in the following way:
\begin{proposition}\label{prop d(P)}
    Let $E/\Q$ be an elliptic curve and $P\in E(\Q)$. Write $P=(\frac{a}{d^2},\frac{b}{d^3})$ and let $d(P):=d$. Then for any $n\geq 2$,
    $$P\in E^0(\Q)\iff d(nP)=|d(P)^{n^2}f_n(P)|$$
\end{proposition}

The formula from (2) of the above theorem is what lets us compute the $p-$adic sigma function. See for example the Algorithm 3.1 on \cite{mazur2006computation}. Implementing that algorithm on a general elliptic curve, we find 
\begin{equation}\label{sigma eq}
\sigma_p(t)=t + \frac{1}{2} a_{1} t^{2} + \left(\frac{3}{8} a_{1}^{2} + \frac{1}{2} a_{2} - \frac{1}{2} c\right) t^{3} + \left(\frac{5}{16} a_{1}^{3} + \frac{3}{4} a_{1} a_{2} - \frac{3}{4} a_{1} c + \frac{1}{2} a_{3}\right) t^{4} + .... 
\end{equation}

We make some observations: 
\begin{enumerate}[label=(\roman*)]
    \item $\sigma_p$ being odd means that $\sigma_p(-P)=-\sigma_p(P)$ for suitable $P\in E(\Q)$. To translate this property on the formal group, we use the isomorphism 
    \begin{equation}\label{formal group eq}
\begin{tikzcd}
	{E^0(\mathbb Q_p)} & {\hat E(\mathbb \mathfrak p)\cong p\mathbb Z_p} \\
	{P=(x,y)=(\frac{a}{d^2},\frac{b}{d^3})} & {t=-\frac{x}{y}=-\frac{a}{b}d}
	\arrow["\sim", from=1-1, to=1-2]
	\arrow[maps to, from=2-1, to=2-2]
\end{tikzcd}
\end{equation}
    that sends $P=(x,y)$ to $t=-\frac{x}{y}$. Then $-P=(x,y-a_1x-a3)$ which is not in general sent to $-t$. This explains why there are even terms in \ref{sigma eq}. \\
    However if $a_1=a_3=0$ (e.g. the elliptic curve is in Short Weierstrass form), then $\sigma_p(t)\in t \Z_p[[t]]$ is an odd power series.
    \item To compute the $p-$adic sigma function, one first needs to compute the constant $c$ and in particular the Katz weight 2 $p-$adic  Eisenstein series which is done through Kedlaya's algorithm. The original algorithm required $p\geq 5$ using the minimal model of the elliptic curve. However, by changing the model, Balakrishnan shows how to compute it for $p=3$ as well (see \cite{balakrishnan3adic}).
    \item If $E$ is a CM elliptic curve of rank 1, then the value of $\mathbb E_2(E,\omega)$ is independent of $p$ and belongs to the field of complex multiplication of $E$, by \cite[Lemma~8.0.13]{katz1976E2}
\end{enumerate}
\subsection{The ordinary case}
We can now describe the construction of the $p$-adic height. We skip the theoretical background and only focus on the computational aspect. We follow \cite{mazur2006computation}:

\textbf{Definition:} Let $E$ be an elliptic curve over $\Q$ with good ordinary reduction at an odd prime $p$ and $P\in E^1(\Q_p)\cap E^0(\Q)$. The (cyclotomic) $p-$adic height is defined as \begin{equation}\label{pheight}
    h_p(P)= \log_p \frac{\sigma_p(P)}{d(P)}
\end{equation}
where $\log_p$ is the $p-$adic logarithm. \\ 
If $P\in E(\Q)$, we can find $m\in \N$ such that $Q:=mP \in E^1(\Q_p)\cap E^0(\Q)$. We then define $$h_p(P):= \frac{1}{m^2}h_p(Q)$$
We mention without proof that this definition is independent of the choice of Weierstrass model of the elliptic curve.

\textbf{Remark:} 
\begin{enumerate}
\item  If $E$ is given in the \textit{global minimal Weierstrass model}, then one can pick $$ m=\mathrm{lcm}(\mathrm{Tam}(E),N_p) $$  \\
The idea behind this is that $\mathrm{Tam}(E)\cdot P\in E^0(\Q)$ for all $P\in E(\Q)$ by definition of the Tamagawa numbers and $N_p=[E^0(\Q_p):E^1(\Q_p)]=[E(\Q_p):E^1(\Q_p)]$ since $p$ is a good prime.

\item If $E$ is NOT given in a global minimal Weierstrass model, then $\mathrm{Tam}(E)\cdot P$ will no longer necessarily land in $E^0(\Q)$. To find a suitable scalar multiple, we can use Theorem \ref{tamagawa theorem} in order to evaluate the index $[E(\Q):E^0(\Q)]$ in terms of the Tamagawa numbers. This index is of course independent of the point $P$. Alternatively, we can use Proposition \ref{prop u_2(P) is unit} to find the smallest scalar multiple $m$ such that $mP\in E^0(\Q)$. This depends on the point $P$, but is computationally more efficient.
\end{enumerate}

\subsection{Dieudonne Module} \label{Section Dieudonne}
In the supersingular case, there is no canonical $p-$adic height with values in $\Q_p$. Instead one has to use the Dieudonne Module $D_pE$. For our intents and purposes, it suffices to simply characterize the Dieudonne module as a 2-dimensional $\Q_p$-vector space generated by $\omega_E$ and $ \eta=x \omega_E$, where $\omega_E$ is the invariant differential. \\
There is an action of the Frobenius $\phi$ on $D_pE$, with characteristic polynomial $X^2+\frac{1}{p}$ (we are assuming $p\geq 5$ and hence that $a_p=0$). We  write
 $$\phi=\begin{bmatrix}
    \alpha & \beta \\ \gamma & \delta
\end{bmatrix}_{(\omega,\eta)}$$
We remark some useful properties that can be found in the proof of Lemme 2.1 in \cite{perrin2003arithmetique}
\begin{proposition}[Perrion-Riou--Bernardi]\label{prop Frob}
    With the above notation:
    \begin{enumerate}
        \item $\gamma \in \Z_p^\times$
        \item  $\alpha \in \Z_p$
        \item  $\ord_p(\beta)=-1$
    \end{enumerate}
\end{proposition}
 It turns out that $(\omega,\phi(\omega))$ is also a basis for $D_pE$. Under this basis, we have that $$\phi= \begin{bmatrix}
    0 & -\frac{1}{p} \\ 1 & 0
\end{bmatrix}_{(\omega,\phi(\omega))}.$$

Finally, simple calculations show that a change of basis gives \begin{equation}
    (A,B)_{(\omega,\eta)}=(A+\frac{\alpha}{\gamma}B, -\frac{1}{\gamma}B)_{(\omega,\phi(\omega))}
\end{equation}

\subsection{The supersingular case}
Now that we have defined the Dieudonne module $D_pE$, we will define for each $v\in D_pE$, a height $$h_v:E(\Q)\to \Q_p$$ which depends linearly on $v$. 
\begin{enumerate}
    \item For $v=\omega$ define \begin{equation}
    h_\omega(P)=\log_{\hat E}^2(P)
    \end{equation} for $P \in E^1(\Q_p)$, where $\log_{\hat E}$ is the elliptic $p$-adic logarithm. One can extend it linearly to all $P\in E(\Q)$. By the identification \begin{equation}
        E^1(\Q_p) \xrightarrow{\sim}p\Z_p
    \end{equation} that sends $P=(x,y)=(\frac{a}{d^2},\frac{b}{d^3})$ to $t=-\frac{x}{y}=-\frac{a}{b}d$, we get \begin{equation}
        h_\omega =t^2 +a_1 t^3 + O(t^4)
    \end{equation}

In particular we remark that \begin{equation}
    \ord_p(h_\omega(P))\geq 2
\end{equation}
and that \begin{equation} \label{height h_omega}
    \ord_p(h_\omega(P))>2 \iff P \in E^2(\Q_p) \iff p^2 |d
\end{equation}

If $P\in E(\Q)-E^1(\Q_p)$ and $p$ is supersingular, the above criterion remains exactly the same since $h_\omega(P)=\frac{1}{N_p}h_\omega(N_p P)$ and $p$ does not divide $N_p$.

\item For $v=\eta$, we define $h_\eta(P)$ in the exact same way as it was defined in the ordinary case (Equation \ref{pheight}).

\end{enumerate}

\subsection{Supersingular p-adic Regulator }
For elliptic curves of rank 1 with supersingular reduction at $p\geq 5$, the $p$-adic regulator is defined as \begin{equation*}
    \mathrm{Reg}_p=(h_\eta,-h_\omega)_{(\omega,\eta)}=(h_\eta+\frac{\alpha}{\gamma}h_\omega, \frac{1}{\gamma}h_\omega)_{(\omega, \phi(\omega))}
\end{equation*} 
as an element in the Dieudonne module $D_pE$.
For a more general definiton over any rank, see \cite{stein2013algorithms}.

A straightforward computation that we will need later gives us \begin{equation}\label{ss regp}
    (1 - \varphi)^2 \text{Reg}_p = \left( \frac{(p\gamma -\gamma) h_\eta + (p \alpha - \alpha + 2) h_\omega}{p \gamma}, \frac{-2 p \gamma h_\eta - (2 p \alpha - p + 1) h_\omega}{p \gamma} \right)_{(\omega, \varphi( \omega))}
\end{equation}

\section{p-adic L-series and the p-adic Birch and Swinnerton-Dyer Conjecture}
In this section, we give a brief overview of the $p-$adic version of the Birch and Swinnerton-Dyer conjecture. \\ 
We write $a_p$ for the trace of Frobenius. Then $x^2-a_px+p$ is the characteristic polynomial. For every root $\alpha$ of the polynomial which satisfies $\ord_p(\alpha)<1$, one can construct a $p$-adic L-series $L_p(E,\alpha,T)$. The construction involves modular symbols and can be found for example in \cite{stein2013algorithms}.

\subsection{The ordinary case}

If $p$ is an ordinary prime, there there is a unique root $\alpha$ with the desired propery and hence a canonical $p-$adic L-series, which we denote $L_p(E,T)$. In fact, it's easy to see that $\alpha \in \Z_p^\times$. This in turn implies that $L_p(E,T)\in \Z_p[[T]]$.

The $p-$adic version of the Birch and Swinnerton-Dyer conjecture in the ordinary case was proposed by Mazur, Tate and Teitelbaum in \cite{mazur-tate-teitelbaum}:
\begin{conjecture}
    Let $E$ be an elliptic curve with good ordinary reduction at a prime $p$.
    \begin{enumerate}
        \item The order of vanishing of $L_p(E,T)$ at $T=0$ is equal to the rank $r=\rank E(\Q)$
        \item Up to a $p-$adic unit, the leading term $L_p^*(E,0)$ satisfies 
        \begin{equation}
            L_p^*(E,0)\sim \frac{ \prod_v c_v \cdot N_p^2 \cdot \# \Sh(E/\Q)}{(\#E(\Q)[p^\infty])^2}\cdot \frac{\Reg_p(E/\Q)}{p}
        \end{equation}
    \end{enumerate}
\end{conjecture}

\subsection{The supersingular case}
In the supersingular case, for $p\geq 5$ the characteristic polynomial becomes $x^2+p$, with both roots $\alpha, \bar \alpha$ satisfying $\ord_p(\alpha)<1$. Furthermore, the series $L_p(E,\alpha,T)$ might not even have integral coefficients in $\Q_p(\alpha)$. However, we can still write $$L_p(E,\alpha,T)=G^+(T) +\alpha G^-(T)$$ with $G^+(T), G^-(T)\in \Q_p[[T]]$.

Perrin-Riou in \cite{perrin2003arithmetique} defined a $p-$adic L-series with coefficients in $D_p(E) \otimes \Q_p[[T]]$, by setting \begin{equation}\label{eq L_p supersingular}
    L_p(E,T)=(G^+(T), -pG^-(T))_{(\omega, \phi(\omega))}
\end{equation}

Then, the $p-$adic Birch and Swinnerton-Dyer Conjecture, according to \cite{bernardiperrinriou1993variante} is:

\begin{conjecture} \label{p-adic BSD supersingular} 
    Let $E$ be an elliptic curve over $\Q$ with supersingular reduction at a prime $p$.
    \begin{enumerate}
        \item The order of vanishing of $L_p(E,T)$ at $T=0$  is equal to the rank $r=\rank E(\Q)$.
        \item The leading term $L_p^*(E,0)$ satisfies 
        \begin{equation} \label{eq leading term supersingular }
            L_p^*(E,0)=\frac{\prod_v c_v \cdot \# \Sh(E\Q)}{(\# E(\Q)_{\text{tor}})^2} \cdot (1-\phi)^2 \frac{\Reg_p(E/\Q)}{p} \in D_p(E)  
        \end{equation}
    \end{enumerate}
\end{conjecture}

Even though $G^+(T)$ and $G^-(T)$ do not have integral coefficients, Pollack, in \cite{pollack2001p}, had the following insight:\\
By using interpolation properties of $G^+(T)$ and $G^-(T)$ already explored by Perrin-Riou in \cite{perrin1990theorie}, one sees that $G^+(T)$ vanishes at $\zeta_{2n}-1$ and $G^-(T)$ vanishes at $\zeta_{2n-1}-1$, where $\zeta_n$ is a $p^n$-th root of unity.  In particular, they have infinitely many zeros. His insight was to construct two $p-$adic power series $\log_p^+(T)$ and $\log_p^-(T)$ that vanish precisely at those zeros. In particular, he defines 
\begin{equation} \label{logp_plus}
    \log_p^+(T)=\frac{1}{p} \prod_{n\geq 1} \frac{\Phi_{2n}(1+T)}{p}
\end{equation}
and \begin{equation} \label{logp_minus}
    \log_p^-(T)=\frac{1}{p}\prod_{n\geq 1} \frac{\Phi_{2n+1}(1+T)}{p}
\end{equation}
where $\Phi_n$ is the $p^n$-th cyclotomic polynomial.

He then defines 
\begin{equation}\label{eq L_p plus}
    L_p^+(E,T):= \frac{G^+(T)}{\log_p^+(E,T)}\qquad \text{and} \qquad L_p^-(E,T):= \frac{G^-(T)}{\log_p^-(E,T)}
\end{equation}
 He proves that these power series have indeed bounded coefficients in $\Z_p$ and hence finitely many zeros. The signed Iwasawa invariants $\mu_p^\pm$ and $\lambda_p^\pm$ refer to the Iwasawa invariants of these two power series.

\section{Main Conjecture}
The main conjecture relates the analytic side, i.e. the $p$-adic $L$-functions, with the algebraic side that comes from the study of Selmer groups. The ordinary and supersingular case have quite a different nature, so we will treat them separately. 
\subsection{The ordinary case}
When $p$ is an ordinary prime, the Pontryagin dual of the Selmer group $\Selp(E/\Q_\infty)$ is a finitely generated, torsion $\Z_p$-module. Therefore, associated to it is a characteristic series $$f_E(T)\in \Z_p[[T]]=\Lambda$$
Schneider (\cite{schneider1985p}) and Perrin-Riou (\cite{perrin1982descente}) showed the following:
\begin{theorem}[Perrin-Riou, Schneider]
    Let $E$ be an elliptic curve over $\Q$ with good ordinary reduction at a prime $p$ and rank $r$. Then $$ord_{T=0}f_E(T) \geq r$$
    It is equal to $r$ if and only if \begin{enumerate}
        \item $\Reg_p(E/\Q) \neq 0$
        \item $\Sh(E/\Q)(p)$ is finite
    \end{enumerate}
    In that case, if we write $f_E(T)=T^r g(T)$, then, up to a $p-$adic unit,
     \begin{equation}
        g_E(0)\sim \frac{\prod_v c_v \cdot N_p^2 \cdot \# \Sh(E/\Q)(p)}{(\#E(\Q)(p))^2} \cdot \frac{Reg_p(E\Q)}{p^r}
    \end{equation}
\end{theorem}

The main conjecture simply states that $L_p(E,T)=f_E(T)\cdot u(T)$, where $u(T)\in \Lambda^\times$
\\ An important property that holds in the ordinary case is Mazur's control theorem:
\begin{theorem}[Mazur's Control Theorem]
    Let $\Gamma=\Gal(\Q_\infty/\Q)$ be the galois group of the cyclotomic $\Z_p$ extension and $E/\Q$ be an elliptic curve with ordinary reduction at $p$. Then $$\Selp(E/\Q) \to \Selp(E/\Q_\infty)^\Gamma$$ has finite kernel and cokernel.
\end{theorem}
In addition, it is well known that $$\ord_{T=0}f_E(T)=\mathrm{corank}_{\Z_p}(\Selp(E/\Q_\infty)^\Gamma)$$
An immediate application of this is: 
\begin{corollary}
    Assume that $\#\Sh(E/\Q)(p)$ is finite. Let $\lambda_p^{MW}$ be the maximum rank of the elliptic curve over the tower of number fields in the cyclotomic $\Z_p$-extension. Then $$\lambda_p^{MW}\leq \lambda_p(E)$$
\end{corollary}
In particular, this means that if $\lambda_p(E)=\rank(E(\Q))$, then the rank of the elliptic curve remains stable in the entire cyclotomic $\Z_p$-extension.

\subsection{The supersingular case}
The supersingular case is a lot more complicated. The Pontryagin dual of the Selmer group is no longer a  torsion $\Lambda$-module. In addition, Mazur's Control theorem is no longer true. 
\\ To overcome these obstacles, one approach is given by Perrin-Riou in \cite{perrin1994theorie}. A more modern approach is given by Kobayashi (\cite{kobayashi2003iwasawa}) where he constructs signed Selmer groups $\Selp^\pm(E/\Q_\infty)$ which are cotorsion $\Lambda$-modules. We will denote $\lambda_p^{\mathrm{alg}}, \mu_p^{\mathrm{alg}}$ the Iwasawa invariants associated to them and  $\lambda_p^{\mathrm{an}}, \mu_p^{\mathrm{an}}$ the ones defined by Pollack's $p-$adic L-functions. We will also denote with $f_E^\pm(T)$ the characteristic series that is associated to them. The main conjecture simply states then that $f_E^\pm(T)$ is generated by Pollack's signed $p$-adic L-functions:
\begin{conjecture}[Main Conjecture]
    $$f_E^\pm(T)=(L_p^\pm(E,T))$$ 
\end{conjecture}

This is proven for CM elliptic curves by Pollack and Rubin in \cite{pollackrubin2004main}. By using Kato's Euler system, Kobayashi proves, that for non-CM elliptic curves and almost all primes $p$:
$$(L_p^\pm(E,T)) \subseteq f_E^{\pm}(T)$$
Finally, there's a version of Mazur's Control theorem in the supersingular case:

\begin{theorem}[Control Theorem]
    The natural map $$\Selp(E/\Q)\to \Selp^{\pm}(E/\Q_\infty)^\Gamma$$ has finite kernel and cokernel.
\end{theorem}
For a proof, see for example \cite{raysujatha2023euler}, Proposition 5.1.

As in the ordinary case, this has the consequence that 
$$\rank(E(\Q_n))\leq \lambda_p^\pm$$
for all $n\geq 1$. Therefore, if either $\lambda_p^+$ or $\lambda_p^-$ is equal to $\rank(E(\Q))$, then the rank of the elliptic curve remains stable in the entire cyclotomic $\Z_p$-extension.

We should clarify at this point that the Iwasawa invariants that are derived from the (signed) Selmer groups are a priori different from those that are derived from the (signed) $p-$adic $L-$series. We call the former algebraic Iwasawa invariants and the latter analytic. Under the main conjecture, they are equal to each other.

\section{Main Results}
We are ready to prove Theorems \ref{ordinary theorem} and \ref{supersingular theorem}. After setting up some notations, we will start with the ordinary case. 
\subsection{Fermat quotient}
To simplify notation, we introduce the following definition:
\begin{definition}
Let $a\in \Z$ and a prime $p$  coprime to $a$. Define $$q_p(a):= \frac{a^{p-1}-1}{p} \mod p$$ the \textit{Fermat quotient}. 
\end{definition}

We list some obvious but useful properties: 
\begin{lemma}\label{lemma 4.1}

\begin{enumerate}
    \item $q_p(1)=0$
    \item $q_p(-a)=q_p(a)$
\end{enumerate}
For $a,b$ coprime to $p$:
\begin{enumerate}[label=(\roman*)]
    \item $q_p(ab)=q_p(a)+q_p(b)$
    \item $q_p(\frac{a}{b})=q_p(a)-q_p(b)$
    \item $q_p(a+np^2)=q_p(a)$.
    \item $q_p(a+ np)=q_p(a) - \frac{n}{a}  $
\end{enumerate}
\end{lemma}

If an elliptic curve has Weierstrass model $E:y^2+a_1xy+a_3y=x^3+a_2x^2+a_4x+a_6 $ and $Q=(\frac{a}{d^2},\frac{b}{d^3})\in E^1(\Q_p)$, then, plugging in the coordinates of the point in the equation of the elliptic curve and clearing denominators, we get: 

\begin{equation}
    b^2+a_1abd+a_3bd^3=a^3+a_2ad^2+a_4d^4+a_6d^6
\end{equation}
which, since $p^2 | d^2$, implies that 
\begin{equation}\label{eq qp of ell curve}
    q_p(b^2+a_1ab\frac{d}{p}p)=q_p(a^3)
\end{equation}

\subsection{The ordinary case}

We start with the following observation. Take an elliptic curve of rank 1 in a global minimal model, and let $P$ be a generator of the free Mordell-Weil group. Using Theorem \ref{Theorem 2.4} and assuming the hypothesis of Theorem \ref{ordinary theorem}, we get that
 $$\mu_p(E)+\lambda_p(E)>1 \iff N_p^2(E) \frac{\Reg_p(E)}{p} \not \in \Z_p^\times \iff \frac{h_p(Q)}{p} \not \in \Z_p^\times $$ 
 where $Q=mP$ is a scalar multiple of $P$ such that $Q\in E^0(\Q)\cap E^1(\Q_p)$ as explained in the previous section.

Now, for an elliptic curve $E:y^2+a_1xy+a_3y=x^3+a_2x^2+a_4x+a_6 $ and a point $Q\in E^1(\Q_p)\cap E^0(\Q)$, we first observe that $$z:=\frac{\sigma_p(Q)}{d(Q)}\in \Z_p^\times$$
\\ Indeed, write $Q=(\frac{a}{d^2},\frac{b}{d^3})$ and identify it with $t=-\frac{a}{b}d$ in $p\Z_p$ under (\ref{formal group eq}) (note that $Q\in E^1(\Q_p)$ implies that $p$ divides $d:=d(Q)$ and $p\nmid a,b$). Therefore, using formula (\ref{sigma eq}), we have $$z=-\frac{a}{b}+\frac{1}{2}a_1\frac{a^2}{b^2}d+ O(d^3) \in \Z_p^\times. $$

\begin{lemma}
Suppose that $z\in \Z_p^\times$ and $N\geq 1$. Then to determine $\log_pz \mod p^N$, it suffices to know $z\mod p^N$. In particular, if $v\in \Z_p^\times$, then $$v_p(\log_pv)>1 \iff v^{p-1}\equiv 1 \mod p^2$$
\end{lemma}

\begin{proof}
    The proof is easy and can be found for example in \cite[Lemma~8]{harvey2008efficient}. 
\end{proof}

Therefore, to determine if $\frac{h_p(Q)}{p} \in \Z_p^\times$, it suffices to know $z\mod p^2$.

\textbf{Case 1: $a_1=0$:} Then \begin{equation} \label{eq 7}
    z=-\frac{a}{b}+O(d^3)
\end{equation} Therefore 
$$\frac{h_p(Q)}{p}\in \Z_p^\times \iff \left(\frac{a}{b}\right)^{p-1}\not\equiv 1 \mod p^2 \iff q_p(a)\neq q_p(b)$$
But by \ref{eq qp of ell curve}, we have that $q_p(b^2)=q_p(a^3)$, i.e. $2q_p(a)=3q_p(b)$. Therefore $$q_p(z)=0 \iff q_p(a)=0 \iff q_p(b)=0$$

\textbf{Case 2: $a_1 \neq 0$:} Then 
$$z=-\frac{a}{b}+\frac{1}{2}a_1\frac{a^2}{b^2}d + O(d^3)$$
If $Q \in E^2(\Q_p)$ (i.e $p^2 |d$), then $q_p(z)=q_p(\frac{a}{b})$ and Equation \ref{eq qp of ell curve} again becomes $2q_p(a)=3q_p(b)$. Therefore, we have the same conclusion as in the previous case

Finally, assume $Q\in E^1(\Q_p)-E^2(\Q_p)$. This is equivalent to saying that $p|d$ but $p^2 \nmid d$. Then:
\begin{equation}\label{eq q_p(z)}
q_p(z)=q_p(\frac{a}{b})+\frac{1}{2}a_1\frac{a}{b}\frac{d}{p}
\end{equation}

On the other hand, by using the properties of the Fermat quotient and rearragning, Equation \ref{eq qp of ell curve} becomes:
\begin{equation}
    a_1\frac{a}{b}\frac{d}{p}=q_p(b^2)-q_p(a^3)
\end{equation}

Substituting into Equation \ref{eq q_p(z)}, we have 
$$q_p(z)=q_p(\frac{a}{b})+q_p(b)-\frac{3}{2}q_p(a)=\frac{1}{2}q_p(a)$$

We conclude that $q_p(z)=0 \iff q_p(a)=0$ which is what we wanted.

This concludes the proof of Theorem \ref{ordinary theorem}
\qed

\subsection{The supersingular case}
The idea for the supersingular primes is simple, although it requires some computational effort: First, we will explore the relationship of the Iwasawa invariants $\mu^\pm, \lambda^\pm$ with the power series $G^\pm$. The $p-$adic BSD conjecture will then give us a relationship between the leading term of $G^\pm$ and the $p-$adic Regulator. 

We start with Pollack's $p-$adic L-functions. Write $$G^+(T)=g_0^+ +g_1^+T +O(T^2), \qquad G^-(T)=g_0^- +g_1^- T +O(T^2)$$ 
with $a_i^\pm \in \Q_p$.
\begin{lemma}\label{Lemma L_p-expansion}
    $$L_p^+=pg_0^+ +(pg_1^++\frac{pg_0^+}{2(p+1)})T +O(T^2)$$ and
    $$L_p^-= pg_0^- +(pg_1^-+\frac{g_0^-}{2(p+1)})T+O(T^2)$$
\end{lemma}

\begin{proof}
    A direct computation gives that $$\frac{\Phi_{r}(1+T))}{p}=1+\frac{p^{r-1}(p-1)}{2}T+O(T^2)$$
    Therefore $$\prod_{n \geq 1} \frac{\Phi_{2n}(1+T)}{p}=1+\frac{p-1}{2}(p+p^3+p^5+...)T+O(T^2)=1-\frac{p}{2(p+1)}T$$
    and 
    $$ \prod_{n \geq 1} \frac{\Phi_{2n+1}(1+T)}{p}=1+\frac{p-1}{2}(1+p^2+p^4+...)T+O(T^2)=1-\frac{1}{2(p+1)}T$$

    But if $h(T)=h_0+h_1(T)+O(T^2)$, then $$\frac{1}{h(T)}=\frac{1}{h_0}-\frac{h_1}{h_0^2}T+O(T^2)$$

    Therefore, we have that $$L_p^+=pg_0^+ +(pg_1^++\frac{pg_0^+}{2(p+1)})T +O(T^2)$$ and similarly for the minus part.
\end{proof}

\begin{corollary}
    Let $E$ be an elliptic curve of rank 1 with supersingular reduction at a prime $p$. Then, assuming conjecture \ref{p-adic BSD supersingular}, $$\mu_p^\pm+\lambda_p^\pm=1 \iff \ord_p(g_1^\pm)=-1$$ 
\end{corollary}
\begin{proof}
    Since the elliptic curve has rank 1, we have that $g_0^\pm =0$. The rest follows immediately from Lemma \ref{Lemma L_p-expansion}. 
\end{proof}

Now assume that $p$ does not divide $\prod c_v$ and $\Sh(E/\Q)$. By Equation \ref{eq leading term supersingular }, we have that \begin{equation}
    (g_1^+, pg_1^-)_{(\omega, \phi(\omega))} \sim (1-\phi)^2 \frac{\Reg_p(E/\Q)}{p} 
\end{equation}
up to a $p-$adic unit.
Recall that $\phi=\begin{bmatrix}
    0 & \frac{-1}{p} \\ 1 & 0 \end{bmatrix}_{(\omega,\phi(\omega))}$ and that $$\Reg_p=(h_\eta+\frac{\alpha}{\gamma} h_\omega, \frac{1}{\gamma}h_\omega)_{(\omega, \phi(\omega))}$$

    A direct computation gives that \begin{equation}
        (1-\phi)^2 \Reg_p= \left( \frac{(p\gamma -\gamma) h_\eta + (p \alpha - \alpha + 2) h_\omega}{p^2 \gamma}, \frac{-2 p \gamma h_\eta - (2 p \alpha - p + 1) h_\omega}{p^2 \gamma} \right)_{(\omega, \varphi, \omega)}
    \end{equation}
From Proposition \ref{prop Frob}, recall that $\gamma \in \Z_p^\times, \alpha \in \Z_p$ and $\ord_p(\beta)=-1$. Also recall that $\ord_p(h_\omega)\geq 2$ and $\ord_p h_\eta \geq 1$. 
Therefore, $$\ord_p(g_1^+)=-1 \iff \ord_p(h_\eta)=1$$
But this is the same criterion we had for the ordinary height. If $Q=mP=(\frac{a}{d^2},\frac{b}{d^3})\in E^1(\Q_p)\cap E^0(\Q)$ is the point defined in Theorem \ref{supersingular theorem} we have proven that $$\mu_p^+ + \lambda_p^+=1 \iff a^{p-1}\equiv 1 \mod p^2$$

Now assume that $\mu_p^+ + \lambda_p^+>1$. Therefore, by the above, this is equivalent to $\ord_p(h_\eta)>1$. Then, we see that $$\ord_p(g_1^-)=-1 \iff \ord_p(h_\omega)=2$$

Equation \ref{height h_omega} finishes now the proof of Theorem \ref{supersingular theorem}.
 \qed 

 We remark that from the above proof, one can also get a criterion for $\mu_p^- + \lambda_p^->1$ without requiring the assumption for the plus part, but it would be unwieldy for any practical applications.

\section{Generalized Wieferich primes and computations}
The numerical criterion that we derived in the previous section for the minus part, is reminiscent of the definition of Wieferich primes.
\\
Silverman in \cite{silverman1988wieferich} defines the following: 
\begin{definition}[Silverman] \label{def Wieferich}
Let $A/\Q$ be a commutative algebraic group and let $P\in A(\Q)$ be a point of infinite order. For each prime $p$, let $N_p=\# A(\Z/p\Z)$. A set of non-Wieferich primes for $A$ and $P$ is a set of the form $$W_{A,P}=\{p: N_p P \not\equiv 0 \mod p^2\}$$. 
\end{definition}
    When $A=\mathbb G_m$, this recovers the original definition of generalized non-Wieferich primes, i.e. $$W_{\mathbb G_m, a}=\{p: a^{p-1}\not\equiv 1 \mod p^2\}$$

    When $A=E$ is an elliptic curve of rank 1 and $P$ is a generator of the free Mordell Weil group, then $$W_{E,P}=\{p: N_p P \in E^1(\Q_p)-E^2(\Q_p)\}=\{p: p^2 \nmid d_{N_p}\}$$ where $N_pP=\left(\frac{a_{N_p}}{d_{N_p}^2}, \frac{b_{N_p}}{d_{N_p}^3}\right)$, is exactly the set that appears
in Equation \ref{minus equivalence} of Theorem \ref{supersingular theorem}. Silverman proves that the ABC conjecture implies that $W_{E,P}$ is infinite. Conjecturally, it is of density 1 and its complement is also infinite.

The criterion for the ordinary case as well as the plus part of the supersingular case, leads us to define the following set:
 \begin{definition}
 Let $E$ any elliptic curve over $\Q$ of positive rank and \\ $P=(\frac{a}{d^2},\frac{c}{d^3})\in E(\Q)$ a non-torsion point. Define
 $$\Lambda_{E,P}=\{p: q_p(a_{N_p})\neq 0\}=\{p: a_{N_p}^{p-1}\not\equiv 1 \mod p^2\}$$
 as $p$ ranges over the good primes.
 \end{definition}

If the sequences $\{a_{N_p}\}$ and $\{d_{N_p}\}$ are "random" enough as $N_p$ varies then the probability that $q_p(a_{N_p})=0$ is roughly $\frac{1}{p}$ (and hence the complement of $\Lambda_{E,P}$ is infinite and of density 0) and similarly for $d_{N_p}\equiv 0 \mod p^2$. The intersection would have probability $\frac{1}{p^2}$ and in particular would be finite. 

In practice, this means that there would be finitely many supersingular primes $p$ such that both $\mu_p^+ + \lambda_p^+ >1$ and $\mu_p^- + \lambda_p^- >1$. In particular, that would imply that there are only finitely many supersingular primes $p$ such that the rank of the elliptic curve changes in the cyclotomic $\Z_p$-extension. 

This agrees with the computational evidence, which leads us to the following conjectures:
 \begin{conjecture}\label{conj 2}
 For any elliptic curve $E/\Q$ and $P\in E^0(\Q)$ a non-torsion point:
 \begin{enumerate}
 \item $\Lambda_{E,P}$ has density 1 and its complement is infinite.
 \item $W_{E,P}$ has density 1 and its complement is infinite.
 \item For any elliptic curve $E$ over $\Q$ of rank 1, there are only finitely many primes $p$ of supersingular reduction such that the rank changes in the cyclotomic $\Z_p$-extension.
 \end{enumerate} 
 \end{conjecture}
We should remark however that we 're unable to prove even that $\Lambda_{E,P}$ is non-empty. It would be interesting if one could express the set $\Lambda_{E,P}$ in a more general way, as in definition \ref{def Wieferich}.

 We conclude with a table of the Iwasawa invariants for elliptic curves of rank 1 with supersingular reduction at $p$, computed using  \cite{SageMath}, making use of the theorems \ref{ordinary theorem}, \ref{supersingular theorem}. We remark that conjecturally $\mu_p^\pm=0$ for all supersingular primes $p$. The code and the tables can be found at \cite{my_Github}.

\begin{table}[h]
    \centering
    \begin{tabular}{c c c c c c}
    \hline
    \textbf{p} & \textbf{Total} & \(\lambda^\pm = 1\) & \(\lambda^+ > 1\) & \(\lambda^- > 1\) & \(\lambda^+ \text{ and } \lambda^- > 1\) \\
    \hline
5  & 4399 & 2990 & 746 & 568 & 95 \\
7  & 5051 & 4261 & 291 & 478 & 21 \\
11 & 6627 & 5609 & 580 & 410 & 28 \\
13 & 2166 & 1832 & 164 & 161 & 9 \\
17 & 3435 & 3038 & 197 & 192 & 8 \\
19 & 3689 & 3523 & 17  & 146 & 3 \\
23 & 5978 & 5751 & 28  & 198 & 1 \\
29 & 2928 & 2816 & 18  & 94  & 0 \\
31 & 4119 & 3995 & 6   & 118 & 0 \\
37 & 600  & 589  & 2   & 9   & 0 \\
41 & 2734 & 2665 & 26  & 43  & 0 \\
43 & 1539 & 1510 & 0   & 29  & 0 \\
47 & 5212 & 5129 & 4   & 79  & 0 \\
\hline
    \end{tabular}
    \caption{Elliptic Curves with supersingular reduction at $p$ and their $\lambda$ invariants.}
    \label{tab:lam-conditions-updated}
    \end{table}

\section{Proof of Theorem \ref{tamagawa theorem}}
Finally, we prove Theorem \ref{tamagawa theorem}. We will develop a criterion to check when a point has good reduction and study how is it affected under a change of coordinates.
\subsection{Division Polynomials}

For $P=(\frac{a}{d^2},\frac{b}{d^3})$ we use the following notation: $$nP=\left( \frac{a_n}{d_n^2},\frac{b_n}{d_n^3}\right)=\left(\frac{g_n}{f_n^2},\frac{\omega_n}{f_n^3}\right)=\left(\frac{\hat g_n}{d^2\hat f_n^2},\frac{\hat \omega_n}{d^3 \hat f_n^3}\right)$$
where $g_n, f_n, \omega_n$ are division polynomials
(see \cite{silverman2009arithmetic}( Exercise 3.7 on page 105  ). When there's no confusion, we will write $g_n$ instead of $g_n(P)$ and similarly for the rest.
Note that $\hat f_n(P), \hat g_n(P), \hat \omega_n(P)$ are integers.

We will need the following well known lemma (for example see \cite[Lemma~1]{cheon1998explicit}:
\begin{lemma} \label{lemma 2}
Let $P\in E^1(\Q_p), P=(\frac{a}{d^2},\frac{b}{d^3})$. Then $$v_p(d_n)=v_p(d)+v_p(n)$$
\end{lemma}
The following proposition will play a central role and can be found in \cite[Theorem A]{ayad1992points}:

\begin{proposition}\label{prop E^0(Q_p)}
Let $P\in E(\Q_p)- E^1(\Q_p)$. The following are equivalent:

\begin{enumerate}
    \item $v_p(f_2), v_p(f_3)>0$
    \item $v_p(f_n)>0 \quad \text{for all } n\geq 2$
    \item $\exists m_0\geq 2: v_p(f_{m_0}), v_p(f_{m_{0}+1})>0$
    \item $\exists n_0 \geq 2: v_p(f_{n_0}), v_p(g_{n_0})>0$
    \item $P\in E(\Q)-E^0(\Q_p)$
\end{enumerate}
\end{proposition}

\textbf{Remark:} From the proof, one sees that that $(5) \Rightarrow (3),(4)$ for all \\ $m_0, n_0 \geq 2$.

Finally, we define $$u_n=\frac{d^{n^2}f_n}{d_n}$$ called the \textbf{cancellation term}.
Proposition \ref{prop u_2(P) is unit} is a special case of:
\begin{proposition} \label{prop u_n(P) is unit}
Let $E/\Q$ an elliptic curve and $P \in E(\Q)$. Then for any $n\geq 2$, $$P\in E^0(\Q_p) \iff u_n(P) \in \Z_p^\times$$
\end{proposition}

\begin{proof}
The forward direction is in \cite[Proposition~1]{wuthrich2004adic}. For the opposite direction, we argue as follows:

Assume $P \not \in E^0(\Q_p)$ and $P=(\frac{a}{d^2},\frac{b}{d^3})$. Then $v_p(d)=0$ and $v_p(u_n)=v_p(f_n)-v_p(d_n)$.
\\ $\bullet$ If $nP \not \in E^1(\Q_p)\Rightarrow v_p(d_n)=0$. Hence by Proposition \ref{prop E^0(Q_p)}, we have $v_p(u_n)>0$.\\
$\bullet $ If $nP \in E^1(\Q_p)$, let $v_p(d_n)=n>0$. Note that $d_n | d^{n^2}f_n$ and hence $v_p(f_n) \geq v_p(d_n)$. From the Remark after Proposition \ref{prop E^0(Q_p)}, we have that $v_p(g_n)>0$. Hence $v_p(f_n)>v_p(d_n)$, which completes the proof.
\end{proof}
Varying through all the primes $p$ gives us Proposition \ref{prop d(P)}
\subsection{Change of Coordinates}

Consider an elliptic curve $E/\Q$ with Weierstrass model $E:y^2+a_1xy+a_3y=x^3+a_2x^2+a_4x+a_6 $. A general change of coordinates is of the form $[u,r,s,t]$ with $x=u^2 x'+r$ and $y=u^3y'+u^2sx'+t$. Since this can always be decomposed into the change of coordinates $[u,0,0,0]$ and $[0,r,s,t]$ and the latter has no effect in the discriminant, it suffices for our purpose to only study the change of coordinates of the form $$(x,y)\to (u^2x,u^3y)$$ and we will denote it as $[u,0,0,0]$. 
\\
For the rest of the section, we fix a prime $p$ and denote $v=v_p$ its normalized valuation.
\\
An immediate corollary is the following:
\begin{corollary}\label{cor change of weier model}
Let $[u,0,0,0]:E\to E'$. Then:
\begin{enumerate}
    \item $\Delta'=u^{12} \Delta$.  In particular if $u=p^r$ then $r=\frac{v(\Delta)'-v(\Delta)}{12}$.
    \item $a_i'=u^ia_i$
    \item $f_2'=u^3f_2$ and $f_3'=u^8f_3$.
    
\end{enumerate}

\end{corollary}
\begin{proof}
    (1) and (2) are immediate from the definitions. For (3) one has to use the general formula of the division polynomials, which is less common: 
    $$f_2(x,y)=2y+a_1x+a_3$$
    $$f_3(x,y)=3x^4+b_2x^3+3b_4x^2+3b_6x+b_8$$
    where $b_2=a_1^2+4a_2, b_4=2a_4+a_1a_3, b_6=a_3^2+4a_6, b_8=a_1^2a_6-a_1a_3a_4+a_2a_3^2-4a_2a_6-a_4^2$.\\
    The rest is just computations. 
\end{proof}

\subsection{Proof of Theorem \ref{tamagawa theorem}}

First we start with some observations:\\
Consider the change of coordinates $[p^{r},0,0,0]: E \to E'$. If $P=(x,y) \in E(\Q)$ and $P'=(p^{-2r}x, p^{-3r}y) \in E(\Q)$ its image, then by Corollary \ref{cor change of weier model} we get that 
\begin{equation} \label{equation 4}
    v(f_2')=3r +v(f_2)  
\end{equation}
\begin{equation} \label{equation 5}
  v(f_3') = 8r + v(f_3)  
\end{equation}

where $f_2=f_2(P), f_2'=f_2(P')$.

Also, if $P=\left(\frac{a}{d^2},\frac{b}{d^3}\right) \in E^1(\Q_p)$, then by Lemma \ref{lemma 2} we have $$v(d_m)=v(d)+v(m).$$ 

and since $u_m(P)\in \Z_p^\times$ by Proposition $\ref{prop u_n(P) is unit}$, we get \begin{equation} \label{equation 1}
    v(f_m)=v(d_m)-m^2 v(d)=(1-m^2) v(d) + v(m)
\end{equation}

We end up with: 
\begin{equation} \label{equation 2}
    v(f_2)=-3v(d) + v(2) = \begin{cases}
    -3v(d), & p \neq 2 \\ 
    -3v(d) +1, & p=2
    \end{cases}
\end{equation}

\begin{equation} \label{equation 3}
    v(f_3)=-8v(d) + v(3) = \begin{cases}
    -8v(d), & p \neq 3 \\ 
    -8v(d) +1, & p=3
    \end{cases}
\end{equation}

With these prerequisites out of the way, take $P\in E^r(\Q_p)$ and $P'$ its image in $E'$. In other words, $v(d(P)) \geq r$. Therefore, by (\ref{equation 2}) and (\ref{equation 3}) we have that at least one of the following must be true: 

\begin{equation}\label{equation 6}
    v(f_2(P)) \leq -3r
\end{equation}

\begin{equation} \label{equation 7}
    v(f_3(P)) \leq -8r
\end{equation}

 Write $f_2'=f_2(P')$.By equations (\ref{equation 4}), (\ref{equation 5}), (\ref{equation 6}) and (\ref{equation 7}), we conclude that $v(f_2') \leq 0$ or $v(f_3') \leq 0$ (or both). By Proposition \ref{prop E^0(Q_p)} we conclude that $P' \in E'^0 (\Q_p)$. We proved that if $P\in E^r(\Q_p)$, then its image $P'$ lies in $ E'^0(\Q_p)$.

It is clear that this map is injective, and that its inverse $$[p^{-r},0,0,0]:E'^0(\Q_p) \to E^r(\Q_p)$$ is well defined and injective as well. This completes the proof of Theorem \ref{tamagawa theorem}.
\qed
\vspace{0.2in}

We can now apply this to prove Corollary \ref{cor short weier}:

Let $E_{min}/\Q$ an elliptic curve given in a minimal model and $E/\Q$ the elliptic curve in short Weierstrass form. For any prime other than 2 or 3, $E/\Q$ is also $p-$minimal. For the primes $p=2$ and $p=3$, there's always a change of coordinates $[\frac{1}{p},r,s,t]:E_{min}\to E$ \cite[Chapter~III]{silverman2009arithmetic}. Hence, if $E$ is not minimal at $p$, then $$E^0(\Q_p)\simeq E^1_{min}(\Q_p)$$ and $$[E(\Q_p):E^0(\Q_p)]=[E_{\min}(\Q_p):E^1_{\min}(\Q_p)]=c_pN_p$$
\qed

Note that  by the Weil bound, we have $N_2 \leq 5$ and $N_3 \leq 7$. Therefore, $[E_{\text{short}}(\Q_p):E^0_{\text{short}}(\Q_p)]$ will always divide $7!\cdot c_p$

We summarize the results of this section in the following diagram:

\[\begin{tikzcd}
	{E^n(\mathbb Q_p)} && {E_{min}^{n+r}(\mathbb Q_p)} \\
	{E^0(\mathbb Q_p)} && {E_{min}^r(\mathbb Q_p)} \\
	&& {E_{min}^1(\mathbb Q_p)} \\
	&& {E_{min}^0(\mathbb Q_p)} \\
	{E(\mathbb Q_p)} && {E_{min}(\mathbb Q_p)}
	\arrow["\sim"', tail, from=2-3, to=2-1]
	\arrow["{[p^{r},0,0,0]}"', from=5-3, to=5-1]
	\arrow[from=2-1, to=1-1, no head]
	\arrow["\sim"', tail, from=1-3, to=1-1]
	\arrow[from=5-1, to=2-1, no head]
	\arrow["{c_p}"', from=5-3, to=4-3, no head]
	\arrow["{N_p}"', from=4-3, to=3-3, no head]
	\arrow["{p^{r-1}}"', from=3-3, to=2-3, no head]
	\arrow["{p^n}"', from=2-3, to=1-3, no head]
	\arrow["{c_pN_pp^{n+r-1}}"', shift right=4, curve={height=30pt}, from=5-3, to=1-3, no head]
\end{tikzcd}\]

\bibliography{mybib}


\end{document}

%% file: my_preamble.tex
\usepackage{quiver}
\usepackage{amsmath, enumitem,amsthm,mathtools,amssymb,eufrak,thmtools}
\usepackage{xcolor, subcaption, qtree, tikz,tikz-cd}
\usetikzlibrary{arrows,trees,positioning}
\usepackage[utf8]{inputenc}
\usepackage[colorinlistoftodos]{todonotes}
\usepackage{comment}

\newtheorem{theorem}{Theorem}[section]
\newtheorem{corollary}{Corollary}[section]
\newtheorem{lemma}{Lemma}[section]
\newtheorem{proposition}{Proposition}[section]
\newtheorem*{theorem*}{Theorem}
\newtheorem*{corollary*}{Corrolary}
\newtheorem{conjecture}{Conjecture}[section]

\newtheorem{definition}{Definition}[section]

\newtheorem{example}{Example}[section]

\usepackage{hyperref}
\definecolor{darkgoldenrod}{rgb}{0.72, 0.53, 0.04}
\definecolor{vegasgold}{rgb}{0.77, 0.7, 0.35}
\definecolor{gold(metallic)}{rgb}{0.83, 0.69, 0.22}
\definecolor{sepia}{rgb}{0.44, 0.26, 0.08}
\definecolor{silver}{rgb}{0.75, 0.75, 0.75}
\hypersetup{
 colorlinks=true,
 linkcolor=gold(metallic),
 filecolor=sepia,      
 urlcolor=vegasgold,
 citecolor=darkgoldenrod,
    }

\usepackage[numbers]{natbib}

%% file: my_macros.tex
\newcommand{\N}{\mathbb{N}}

\newcommand{\Z}{\mathbb{Z}}

\newcommand{\Q}{\mathbb{Q}}
\newcommand{\F}{\mathbb{F}}

\newcommand{\ot}{\mathcal{O}}

\newcommand{\Selp}{\Sel_{p^\infty}}

\DeclareMathOperator{\Gal}{Gal}
\DeclareMathOperator{\Sel}{Sel}

\DeclareMathOperator{\ord}{ord}
\DeclareMathOperator{\rank}{rank}
\DeclareMathOperator{\Reg}{Reg}
\DeclareMathOperator{\Coker}{Coker}
\DeclareMathOperator{\Tam}{Tam}

\DeclareFontFamily{U}{wncy}{}
\DeclareFontShape{U}{wncy}{m}{n}{<->wncyr10}{}
\DeclareSymbolFont{mcy}{U}{wncy}{m}{n}
\DeclareMathSymbol{\Sh}{\mathord}{mcy}{"58}